\documentclass[11pt,a4paper,twoside]{amsart}
\usepackage{amsfonts,amssymb,amscd,amsmath,amsthm, enumerate}
\usepackage{graphicx}
\title[Preservers of absolutely compatible pairs]{Contractive linear preservers of absolutely compatible pairs between C$^*$-algebras}

\author{Nabin K. Jana, Anil K. Karn}
\address{School of Mathematical Science, National Institute of Science Education and
Research, HBNI, Bhubaneswar, At \& Post - Jatni, PIN - 752050, India.}
\email{nabinjana@niser.ac.in; anilkarn@niser.ac.in}

\author[A.M. Peralta]{Antonio M. Peralta}
\address{Departamento de An{\'a}lisis Matem{\'a}tico, Facultad de
Ciencias, Universidad de Granada, 18071 Granada, Spain.}
\email{aperalta@ugr.es}



\date{}

\newtheorem{theorem}{Theorem}[section]

\newtheorem{proposition}[theorem]{\bf Proposition}
\newtheorem{corollary}[theorem]{\bf Corollary}
\newtheorem{definition}[theorem]{\bf Definition}

\newtheorem{remark}[theorem]{\bf Remark}

\usepackage[colorlinks, bookmarks=true]{hyperref}
\usepackage{color,graphicx,shortvrb}
\usepackage[active]{srcltx} %SRC Specials for DVI search

\begin{document}

\begin{abstract} Let $a$ and $b$ be elements in the closed ball of a unital C$^*$-algebra $A$ (if $A$ is not unital we consider its natural unitization). We shall say that $a$ and $b$ are domain (respectively, range) {absolutely compatible} {\rm(}$a\triangle_d b$, respectively, $a\triangle_r b$, in short{\rm)} if $\Big| |a| -|b| \Big| + \Big| 1-|a|-|b| \Big| =1$ (respectively, $\Big| |a^*| -|b^*| \Big| + \Big| 1-|a^*|-|b^*| \Big| =1$), where $|a|^2= a^* a$. We shall say that $a$ and $b$ are {absolutely compatible} {\rm(}$a\triangle b$ in short{\rm)} if they are both range and domain absolutely compatible. In general, $a\triangle_d b$ (respectively, $a\triangle_r b$ and $a\triangle b$) is strictly weaker than $ab^*=0 $ (respectively, $a^* b =0$ and $a\perp b$). Let $T: A\to B$ be a contractive linear mapping between C$^*$-algebras. We prove that if $T$ preserves domain absolutely compatible elements {\rm(}i.e., $a\triangle_d b\Rightarrow T(a)\triangle_d T(b)${\rm)} then $T$ is a triple homomorphism. A similar statement is proved when $T$ preserves range absolutely compatible elements. It is finally shown that $T$ is a triple homomorphism if, and only if, $T$ preserves absolutely compatible elements.
\end{abstract}

\maketitle

\section{Introduction}

In a C$^*$-algebra $A$ we can consider two natural relations among pairs of elements these notions are ``zero product'' and ``orthogonality''. Elements $a,b$ in $A$ have \emph{zero product} (respectively, \emph{zero products}) if $a b=0$ (respectively, $a b = ba =0$), while $a$ and $b$ are called \emph{orthogonal} ($a\perp b$ in short) if $a b^* = b^* a =0$. Being orthogonal is a symmetric relation. A couple of non-symmetric relations can be defined by $a b^*=0$ and $b^* a = 0$, respectively.\smallskip

Orthogonality is defined in terms of the product and involution of the C$^*$-algebra. Several characterizations of orthogonality can be found, for example, in \cite[Corollary 8]{FerMarPe2012} and \cite[Theorem 2.1]{Karn2018}.
%There exist several characterizations of orthogonality. For example, it is shown in \cite[Corollary 8]{FerMarPe2012} that if $a$ and $b$ are elements in a C$^*$-algebra $A$ with $\|a\|<\sqrt{2}$, then $a\perp b$ if, and only if, $B(a,a) (b) =b$, where $B(a,a)$ denotes the \emph{Bergmann operator} associated with $a$ which is defined by $$B(a,a)(z) = (1-aa^*) z(1-a^*a)\ \ (\forall z\in A).$$ Additional characterizations of orthogonality for positive elements in the closed unit ball, $\mathcal{B}_{A}$, of a C$^*$-algebra has been recently established in \cite{Karn2018}. Theorem 2.1 in \cite{Karn2018} shows that if $a, b \in A^+\backslash\{0\},$ then $ab = 0$ if, and only if, for every $0 < c \leq a$ and $0 < d \leq b$ in $A$ we have $\| \|c\|^{-1} c + \| d\|^{-1} d\| = 1$.
For an order-theoretic approach, we present some notation. Given an element $a$ in a C$^*$-algebra $A$, we shall write $|a| = (a^* a)^{1/2}$ for the absolute value of the element $a$ in $A$. Let $a$ and $b$ be two positive elements in the closed unit ball of a unital C$^*$-algebra $A$, then by \cite[Proposition 4.1]{Karn2018}, $a$ is orthogonal to $b$ if, and only if, the following two conditions hold:\begin{enumerate}[$(O.1)$]\item $a+b\leq 1$;
	\item $\vert a - b \vert + \vert 1 - a - b \vert = 1.$
\end{enumerate}

The second condition above gives rise to the notion of absolute compatibility. Following the notation introduced in \cite{Karn2018}, given two positive elements $a$, $b$ in the closed unit ball of a unital C$^*$-algebra $A$, we shall say that $a$ is \emph{absolutely compatible} with $b$ ($a\triangle b$ in short) if $\vert a - b \vert + \vert 1 - a - b \vert = 1$. \smallskip

Clearly, the relation ``being absolutely compatible to'' is symmetric. By definition if $0\leq a$ is absolutely compatible with $0\leq b$ (with $a,b\in [0,1]_{A}= \{ x\in A: 0\leq x\leq 1 \}$), then $\{a, 1 - a \}$ is absolutely compatible with $\{ b, 1 - b \}$, where a set $S\subseteq A$ is said to be absolutely compatible with another set $T\subseteq A$ if each element in $S$ is absolutely compatible with every element of $T$.\smallskip

It is shown in \cite[Proposition 4.9]{Karn2018} that if $p$ is a projection and $a$ is an arbitrary positive element in the closed unit ball of a unital C$^*$-algebra, then $a\triangle p$ if, and only if, $a$ and $p$ commute. However, absolute compatibility is not, in general, equivalent to commutativity nor to orthogonality. For example, the elements $a = \left( \begin{array}{cc} \frac{2}{3} & \frac{1}{3} \\ \frac{1}{3} & \frac{1}{3} \end{array}
\right)$ and $b = \left( \begin{array}{cc} \frac{2}{3} & - \frac{1}{3} \\ - \frac{1}{3} & \frac{1}{3}
\end{array} \right)$ are positive and absolutely compatible in $M_2(\mathbb{C})$, with $a b \neq b a\neq 0$. The whole picture is well explained by the following strict inclusion $$\{ (a,b)\in A^2 : 0\leq a,b \leq 1,\ a\perp b  \} \subsetneqq \{ (a,b)\in A^2 : 0\leq a,b \leq 1,\ a\triangle b  \}.$$

Continuous linear maps between C$^*$-algebras preserving orthogonal elements have been widely studied. The pioneering study of M. Wolff in  \cite{Wolff94} provides a complete description of those continuous linear and symmetric operators between C$^*$-algebras preserving orthogonal elements.
%The first result in the literature concerning  non-commutative C$^*$-algebras is due to M. Wolff. Suppose $T: A\to B$ is a bounded linear operator between
%two C$^*$-algebras and assume that $A$ is unital. We further assume that $T$ is symmetric (i.e. $T(a)^* = T(a^*)$ for all $a
%\in A$) and preserves orthogonality (i.e. $a\perp b$ in $A$ $\Rightarrow$ $T(a)\perp T(b)$). Then denoting $T(1) = h$ the
%following assertions hold:
%\begin{enumerate}[{\rm $a)$}] \item $T(A)$ is contained in the
%	norm closure of $h \{h\}^{'},$ where $\{h\}^{'}$ denotes the
%	commutator of $\{h\}$. \item There exists a Jordan *-homomorphism
%	$S: A \to B^{**}$ such that $T(z) = h S(z)$, for all $z\in
%	A$\end{enumerate} (see \cite[Theorem 2.3]{Wolff94}).
\smallskip

Ng.-Ch. Wong proves in \cite[Theorem 3.2]{Wong2005} that a continuous linear operator $T$ between two C$^*$-algebras is a triple homomorphism (i.e. a linear mapping preserving triple products of the form $\{a,b,c\} = \frac12 (ab^* c + cb^* a)$) if, and only if, $T$ is orthogonality preserving and $T^{**} (1)$ is a partial isometry. It should be remarked here that a triple homomorphism between C$^*$-algebras (or, more generally, between JB$^*$-triples) is automatically continuous and contractive (see \cite[Lemma 1]{BarDanHorn}).\smallskip

Continuous linear maps between C$^*$-algebras preserving pairs of orthogonal elements have been completely determined in \cite{BurFerGarMarPe2008} (see \cite{BurFerGarPe2009} for a Jordan version). Briefly speaking, by \cite[Theorem 17 and Corollary 18]{BurFerGarMarPe2008}, given a bounded linear mapping $T$ between two C$^*$-algebras $A$ and $B$, then denoting $h= T^{**} (1)$ and $r= r(h)$ the range partial isometry of $h$ in $B^{**}$, it follows that $T$ is orthogonality preserving if, and only if, one of the equivalent statements holds:
\begin{enumerate}[{\rm $(a)$}] \item There exists a unique triple homomorphism $S: A
		\to B^{**}$ satisfying $h^* S(z) = S(z^*)^* h,$ $h S(z^*)^* = S(z)
		h^*,$ and $$T(z) = h r(h)^* S(z)=S(z) r(h)^* h, \hbox{ for all } z\in A;$$
		\item $T$ preserves zero triple products, that is, $\{{T(x)},{T(y)},{T(z)}\} = 0$ whenever $\{x,y,z\} =0$.
\end{enumerate} The hypothesis concerning continuity can be relaxed at the cost of assuming alternative assumptions, for example, every linear surjection $T$ between von Neumann algebras preserving orthogonality in both directions (i.e. $a\perp b$ $\Leftrightarrow$ $T(a)\perp T(b)$) is automatically continuous (see \cite[Theorem 19]{BurGarPe2011}). Additional results on automatic continuity of linear biorthogonality preservers on some other structures can be seen in \cite{OikPe2013,OikPePu2013,GarPe2014,Pe2017} and \cite{LiuChouLiaoWong2018}. For additional details on orthogonality preservers the reader is referred to the recent survey \cite{LiuChouLiaoWong2018b}.\smallskip

The previous results show that bounded linear operators $T$ which are orthogonality preserving are nothing but appropriate multiples of Jordan $^*$-homomorphisms after considering an alternative Jordan product in the target C$^*$-algebra.\smallskip

In view of the characterization of orthogonality established by conditions $(O.1)$ and $(O.2)$, the problem of determining those contrractive linear operators between C$^*$-algebras preserving absolutely compatible elements seems a natural step forward in the study of linear preservers.
First, we need to extend the notion of absolute compatibility to hermitian and general elements.

Let $a,b$ be elements in the closed ball of a unital C$^*$-algebra $A$ (if $A$ is not unital we consider its natural unitization). We shall say that $a$ and $b$ are \emph{domain} (respectively, \emph{range}) \emph{absolutely compatible} {\rm(}$a\triangle_d b$, respectively, $a\triangle_r b$, in short{\rm)} if $|a|$ and $|b|$ (respectively, if $|a^*|$ and $|b^*|$) are absolutely compatible, that is, $\Big| |a| -|b| \Big| + \Big| 1-|a|-|b| \Big| =1$ (respectively, $\Big| |a^*| -|b^*| \Big| + \Big| 1-|a^*|-|b^*| \Big| =1$). Finally, $a$ and $b$ are called \emph{absolutely compatible} {\rm(}$a\triangle b$ in short{\rm)} if they are range and domain absolutely compatible. Having this new notions in mind, we show that the following statements hold whenever $a$ and $b$ are elements in the closed unit ball of $A$:
\begin{enumerate}[$(a)$]
\item $a b^*  = 0$ if, and only if, $|a| + |b| \leq 1$ and $a\triangle_d b$ if, and only if, $|a| + |b| \leq 1$ and $a^*\triangle_r b^*$;
\item $b^* a =0$ if, and only if, $|a^*| + |b^*| \leq 1$ and $a^*\triangle_d b^*$ if, and only if, $|a^*| + |b^*| \leq 1$ and $a\triangle_r b$;
\item $a\perp b$ if, and only if, $|a| + |b| \leq 1,$ $|a^*| + |b^*| \leq 1,$  $a\triangle_r b$, and $a\triangle_d b$;
\end{enumerate} When $a$ and $b$ are self adjoint, then $a\perp b$ if, and only if, $|a| + |b| \leq 1,$ and $a\triangle b$ (see Proposition \ref{p characterization orthogonality for general elements}). These statements generalize the characterization in $(O.1)$ and $(O.2)$.\smallskip

Among our main results we prove that an element $a$ in the closed unit ball of a C$^*$-algebra is a projection if, and only if, $a\triangle a$. In particular a positive element $a$ in $A$ with $a\leq 1$ is a projection if, and only if, $a\triangle a$. This is the starting point to show that every contractive linear operator $T$ between von Neumann algebras preserving domain absolutely compatible elements {\rm(}i.e., $a\triangle_d b\Rightarrow T(a)\triangle_d T(b)${\rm)} or range absolutely compatible elements {\rm(}i.e., $a\triangle_r b\Rightarrow T(a)\triangle_r T(b)${\rm)} is a triple homomorphism (see Proposition \ref{p presevers on von Neumann} and Corollary \ref{c r presevers on Cstar}). The reciprocal implication is, in general, false (cf. Remark \ref{r examples star hom}).\smallskip

The existence of C$^*$-algebras admitting no non-trivial projections makes the problem of extending the above conclusion to contractive operators between general C$^*$-algebras a more difficult technical goal, which is finally established in Theorem \ref{t presevers on Cstar}. We finally prove that a contractive linear operator between two C$^*$-algebras preserves absolutely compatible elements {\rm(}i.e., $a\triangle b\Rightarrow T(a)\triangle T(b)${\rm)} if, and only if, $T$ is a triple homomorphism (see Corollary \ref{c AC presevers on Cstar}). It should be added here that, a priori, there exists no obvious connection between linear orthogonality preservers and absolute compatibility preservers.\smallskip

\section{Absolutely compatible elements in the closed unit ball of a C$^*$-algebra}

The range projection of an element $a$ in a von Neumann algebra $M$ is the smallest projection $p$ in $M$ satisfying $p a =a$. The range projection of $a$ will be denoted by $r(a)$ (see, for example, \cite[2.2.7]{Ped}).\smallskip

Henceforth the positive cone of a C$^*$-algebra $A$  will be denoted by $A^+$.\smallskip

Let $a$ and $b$ be elemnets in a C$^*$-algebra $A$. Following the notation in \cite{LiuChouLiaoWong2018b,LiuChouLiaoWong2018}, we shall say that $a$ and $b$ are \emph{range orthogonal} (respectively, \emph{domain orthogonal}) if $a^*b =0$, or equivalently, $|a^*| \perp |b^*|$  (respectively, if $ab^*=0$, or equivalently, $|a| \perp |b|$). Clearly, $a\perp b$ if, and only if, $a$ and $b$ are range and domain orthogonal (that is, $|a|\perp |b|$ and $|a^*|\perp |b^*|$). In order to keep a consistent notation we introduce the next definitions.

\begin{definition}\label{def range and domain AC} Let $a,b$ be elements in the closed ball, $\mathcal{B}_{A}$, of a unital C$^*$-algebra $A$. We shall say that $a$ and $b$ are \emph{domain absolutely compatible} {\rm(}$a\triangle_d b$ in short{\rm)} if $|a|$ and $|b|$ are absolutely compatible, that is, $$\Big| |a| -|b| \Big| + \Big| 1-|a|-|b| \Big| =1.$$ We shall say that $a$ and $b$ are \emph{range absolutely compatible} {\rm(}$a\triangle_r b$ in short{\rm)} if $|a^*|$ and $|b^*|$ are absolutely compatible, that is, $$\Big| |a^*| -|b^*| \Big| + \Big| 1-|a^*|-|b^*| \Big| =1.$$ Finally, $a$ and $b$ are called \emph{absolutely compatible} {\rm(}$a\triangle b$ in short{\rm)} if they are range and domain absolutely compatible.
\end{definition}

If $A$ is not unital we shall say that $a$ and $b$ are (\emph{range, domain}) \emph{absolutely compatible} if they are (range, domain) absolutely compatible in the unitization of $A$. Obviously, $a\triangle_d b$ if, and only if, $a^*\triangle_r b^*$.\smallskip

\begin{remark}\label{r new}{\rm Suppose $A$ is a commutative unital C$^*$-algebra of the form $A= C(\Omega)$, where $\Omega$ is a compact Hausdorff space. In this setting, given $f,g\in \mathcal{B}_{A}$, we have $f\triangle g$ if, and only if, $f\triangle_d g$  if, and only if, $f\triangle_r g$. Furthermore, $f\triangle_d g$ if, and only if, the following properties hold:\begin{enumerate}[$(a)$]\item $f(t) g(t) =0 $ for every $t\in \Omega$ such that $|f(t)|,|g(t)|<1$;
\item If $|g(t)| = 1$, for a certain $t\in \Omega$, then $f(t)$ can take any value in $\mathbb{B}_{\mathbb{C}}$;
\item If $|f(t)| = 1$, for a certain $t\in \Omega$, then $g(t)$ can take any value in $\mathbb{B}_{\mathbb{C}}$.
	\end{enumerate}}
\end{remark}

After widen the notion of absolute compatibility to general elements in the closed unit ball of a C$^*$-algebra, our next goal is a extension of the characterization of orthogonality established in \cite[Proposition 4.1]{Karn2018}.

\begin{proposition}\label{p characterization orthogonality for general elements}
	Let $a$ and $b$ be two elements in the closed unit ball of a unital C$^*$-algebra $A$. Then the following statements hold:
	\begin{enumerate}[$(a)$]
		\item $a b^*  = 0$ if, and only if, $|a| + |b| \leq 1$ and $a\triangle_d b$ if, and only if, $|a| + |b| \leq 1$ and $a^*\triangle_r b^*$;
		\item $b^* a =0$ if, and only if, $|a^*| + |b^*| \leq 1$ and $a^*\triangle_d b^*$ if, and only if, $|a^*| + |b^*| \leq 1$ and $a\triangle_r b$;
		\item $a\perp b$ if, and only if, $|a| + |b| \leq 1,$ $|a^*| + |b^*| \leq 1,$  $a\triangle_r b$, and $a\triangle_d b$;
	\end{enumerate} When $a$ and $b$ are self adjoint, then $a\perp b$ if, and only if, $|a| + |b| \leq 1,$ and $a\triangle b$.
\end{proposition}

\begin{proof} Clearly $(c)$ follows from $(a)$ and $(b)$, and $(b)$ follows from $(a)$ by just replacing $a$ and $b$ with $a^*$ and $b^*$, respectively.\smallskip
	
	$(a)$ $(\Rightarrow)$ Suppose $a b^*=0$. It is not hard to see that, in this case, we have $|a|\perp |b|$ in $[0,1]$, and hence by \cite[Proposition 4.1]{Karn2018} we get $|a| + |b| \leq 1$ and $|a|\triangle |b|$, and thus $a\triangle_d b$.\smallskip
	
	$(\Leftarrow)$ Suppose $|a| + |b| \leq 1$ and $a\triangle_d b$ (i.e. $|a|\triangle |b|$). It follows from \cite[Proposition 4.1]{Karn2018} that $|a|\perp |b|$, and consequently $ab^* = a r(a^* a) r(b^*b) b^* =0$.
\end{proof}

The following result, which has been borrowed from \cite{Karn2018} shows how the notion of absolute compatibility is related to the natural Jordan product ($a \circ b := \frac{1}{2} (a b + b a)$).

\begin{proposition}\label{p 00}{\rm\cite{Karn2018}}
	Let $A$ be a unital C$^{\ast}$-algebra and let $0\leq a, b \leq 1$ in $A$. Then the following statements are equivalent:
	\begin{enumerate}[$(a)$]
		\item $a$ is absolutely compatible with $b$;
		\item $2 a \circ b = a + b - \vert a - b \vert$;
		\item $a \circ b, (1 - a) \circ (1 - b) \in A^+$ and $( a \circ b ) \Big( (1 - a) \circ (1 - b) \Big) = 0$;
		\item $a \circ (1 - b), (1 - a) \circ b \in A^+$ and $\Big( a \circ (1 - b) \Big) \Big( (1 - a) \circ b \Big) = 0$.
	\end{enumerate}
\end{proposition}

\begin{proof} The equivalences follow from \cite[Propositions 4.5 and 4.7]{Karn2018} and the fact that if  $a$ is absolutely compatible with $b$ then $\{a, 1 - a \}$ is absolutely compatible with $\{ b, 1 - b \}$.
\end{proof}

We shall explore next the consequences of the above results. The first one is a new characterization of projections and partial isometries.

\begin{corollary}\label{c characterization of projections and partial isometries} Let $A$ be a C$^*$-algebra. Then the following statements hold:
	\begin{enumerate}[$(a)$] \item If $a$ is a positive element in $\mathcal{B}_{A}$, then $a$ is a projection if, and only if, $a\triangle a$;
		\item If $a$ is an element in $\mathcal{B}_{A}$, then $a$ is a partial isometry if, and only if, $a\triangle a$.
	\end{enumerate}
\end{corollary}

\begin{proof} We shall only prove $(b)$ because $(a)$ is a consequence of this statement. Suppose $a$ is an element in $\mathcal{B}_{A}$ with $a\triangle a$. By definition $|a|\triangle |a|$, and by Proposition \ref{p 00}, we conclude that $2 |a|^2 = |a| + |a| - \vert |a| - |a| \vert = 2 |a|,$ which assures that $|a|$ is a projection, and hence $a$ is a partial isometry.\smallskip
	
	If $a$ is a partial isometry in $A$, then $|a|$ is a projection, and thus $|a|\triangle |a|$ as desired.
\end{proof}

Let us observe that the recent survey \cite{FerPe18} gathers a wide list of results characterizing partial isometries. The characterization of partial isometries in terms of absolute compatibility offers a new point of view.

\section{Linear contractions preserving absolutely compatible elements}

We can now deal with linear preservers of absolutely compatible elements between von Neumann algebras. Let $T: A\to B$ be a contractive linear mapping between C$^*$-algebras. We shall say that $T$ \emph{preserves absolutely compatible elements} (respectively, domain or range absolutely compatible elements) if $a\triangle b$ in $\mathcal{B}_M$ implies $T(a)\triangle T(b)$ in $\mathcal{B}_N$ (respectively, $a\triangle_d b$ in $\mathcal{B}_M$ implies $T(a)\triangle_d T(b)$ or $a\triangle_r b$ in $\mathcal{B}_M$ implies $T(a)\triangle_r T(b)$). Since Definition \ref{def range and domain AC} is only valid for elements in the closed unit ball, linear absolutely compatible elements must be assumed to be contractive.\smallskip

Suppose $T: A\to B$ preserves absolutely compatible elements. Let us take $a,b\in \mathcal{B}_{A}$ with $a\perp b$. In this case $a\triangle b$ and $|a| +|b|, |a^*|+|b^*|\leq 1$ (cf. Proposition \ref{p characterization orthogonality for general elements}). By hypothesis, $T(a)\triangle T(b)$, and $|T(a)|,|T(a)^*|, |T(b)|, |T(b)^*|\leq 1$.  However, it is not, a priori, clear that $|T(a)| +|T(b)|, |T(a)^*|+|T(b)^*|\leq 1$. So, it is not obvious that $T$ preserves orthogonality.

\begin{proposition}\label{p presevers on von Neumann} Let $T: M\to N$ be a contractive linear operator between von Neumann algebras. Suppose $T$ preserves domain absolutely compatible elements. Then $T$ is a triple homomorphism. Furthermore, if $T$ also is symmetric, then $\Phi = T(1) T$ is a Jordan\hyphenation{Jordan} $^*$-homomorphism and $T =  T(1) \Phi$. If we additionally assume that $T$ is positive {\rm(}i.e. $T(a)\geq 0$ for all $a\geq0${\rm)}, then $T$ is a Jordan $^*$-homomorphism.
\end{proposition}

\begin{proof} We can always assume that $T$ is non-zero. Let $e$ be a partial isometry in $M$. Since $T(e)\in \mathcal{B}_{N},$ it follows that $e\triangle_d e$ (see Corollary \ref{c characterization of projections and partial isometries}). By hypothesis $T$ preserves absolute compatibility, therefore $T(e)\triangle_d T(e),$ and Corollary \ref{c characterization of projections and partial isometries} implies that $T(e)$ is a partial isometry in $N$.\smallskip
	
	Suppose now that $e$ and $v$ are orthogonal partial isometries in $M$. It follows from the above paragraph that $T(e)\pm T(u) = T(e\pm u)$ is a partial isometry in $N$, and consequently $T(e)\perp T(u)$ (compare, for example, \cite[Lemma 3.6]{IsKaRo95}). Suppose $e_1,\ldots,e_m$ are mutually orthogonal partial isometries in $M$ and $\alpha_1,\ldots,\alpha_m$ are positive real numbers. It follows from the above that the element $\displaystyle a= \sum_{j=1}^m \alpha_j e_j$ satisfies $$\{T(a),T(a),T(a)\} = \left\{ T\left(\sum_{j=1}^m \alpha_j e_j\right), T\left(\sum_{j=1}^m \alpha_j e_j\right), T\left(\sum_{j=1}^m \alpha_j e_j\right) \right\} $$ $$= \sum_{j=1}^m \alpha_j^3 T(\{e_j,e_j,e_j\}) = T(\{a,a,a\}).$$ Since the von Neumann algebra $M$ can be regarded as a JBW$^*$-triple in the sense employed in \cite{Horn87}, we can apply a well known result of JB$^*$-triple theory affirming that every element $b$ in $M$ can be approximated in norm by finite (positive) linear combinations of mutually orthogonal partial isometries (also called tripotents) in $M$, to deduce, from the continuity of $T$, that $T(\{b,b,b\}) = \{T(b),T(b),T(b)\}$, for ever $b\in M$. Since $T$ is complex linear, a standard polarization argument shows that $T$ is a triple homomorphism (compare the proof of \cite[Theorem 2.2]{FerMarPe2004}).\smallskip
	
	Suppose now that $T$ is symmetric. In this case, $T(1)$ must be a hermitian partial isometry, and thus $T(1) = p-q$, where $p$ and $q$ are two orthogonal partial isometries in $N$. We also know that $$T(a) = T\{a,1,1\} = \{T(a),T(1),T(1)\} = \frac12 ( T(a) (p+q) + (p+q) T(a) ),$$ and
	$$T(b) = T\{1,b,1\} = \{T(1),T(b),T(1)\} = (p-q) T(b) (p-q) ,$$
	for all $a \in M$, $b\in M_{sa}$. It follows that $ T(a) (p+q) = (p+q) T(a) =  (p+q) T(a) (p+q) =T(a)$ for every $a\in M$, and $T(b) T(1) = T(b) (p-q) = (p-q) T(b) = T(1) T(b)$ for all $b\in M_{sa}$ (and thus, for every $b\in M$). It is not hard to check that, in this case we have $$\Phi (a^2) = T(1) T(a^2) = T(1) T(\{a,1,a\}) = T(1) \{T(a),T(1),T(a)\} $$ $$= T(1) T(a) T(1) T(a) = \Phi (a)^2,$$ for all $a\in M$. The rest is clear.
\end{proof}

We shall see next that the conclusion of our previous Proposition \ref{p presevers on von Neumann} is also true for contractive linear operators preserving domain absolutely compatible elements between general C$^*$-algebras. The existence of (unital) C$^*$-algebras admitting no non-trivial projections makes invalid our previous arguments in the general case. The difficulties in the general case are stronger. Another handicap is that it is not obvious that if $T:A\to B$ is a contractive linear operator preserving domain absolutely compatible elements, then $T^{**}: A^{**}\to B^{**}$ also preserves domain absolutely compatible elements. Henceforth, the spectrum of an element $a$ in a C$^*$-algebra $A$ will be denoted by $\sigma(a)$.\smallskip

Let us recollect some basic results on the strong$^*$ topology of a von Neumann algebra. By Sakai's theorem (see \cite[Theorem 1.7.8]{S}) the product of every von Neuman algebra is separately weak$^*$-continuous, however joint weak$^*$-continuity is, in general, hopeless. Let $M$ be a von Neumann algebra. The \emph{strong$^*$ topology} on $M$ is the locally convex topology on $M$ defined by all seminorms of the form $\|x\|_{\phi}:=\phi(x x^* + x^* x )^{\frac12}$, where $\phi$ runs in the set of positive norm-one normal functional in $M_*$ (see \cite[Definition 1.8.7]{S}). The strong$^*$ topology is stronger than the weak$^*$-topology, and a functional $\varphi: M\to \mathbb{C}$ is weak$^*$-continuous if, and only if, it is strong$^*$-continuous. In particular, linear maps between von Neumann algebras are strong$^*$-continuous if, and only if, they are weak$^*$-continuous \cite[Theorem 1.8.9, Corollary 1.8.10]{S}. The product of every von Neumann algebra is jointly strong$^*$-continuous on bounded sets (cf. \cite[Proposition 1.8.12]{S}).\smallskip

Let $(c_{\lambda})_{\lambda}$ be a net of positive elements in the closed unit ball of a von Neumann algebra $M$ converging to a positive element $c_0\in \mathcal{B}_{M}$ in the strong$^*$ topology. Let $f: [0,1]\to \mathbb{R}$ be a continuous function. Let us fix a norm-one functional $\phi\in M_*$. By Weierstrass approximation theorem, for each $\varepsilon>0$, there exists a polynomial $p$ satisfying $\|f-p\|_{\infty} <\frac{\varepsilon}{3}$. By the joint strong$^*$-continuity of the product of $M$, the net $(p(c_{\lambda}))_{\lambda}$ converges to $p(c_0)$ in the strong$^*$ topology. We can therefore find $\lambda_0$ such that $|\phi (p(c_{\lambda})-p(c_0))| <\frac{\varepsilon}{3}$ for every $\lambda \geq \lambda_0$, and consequently, $$ |\phi (f(c_{\lambda})-f(c_0))| \leq 2 \|f-p\|_{\infty} + |\phi (p(c_{\lambda})-p(c_0))| <\varepsilon,$$ for every $\lambda \geq \lambda_0$. This shows that $(f(c_{\lambda}))_{\lambda}\to f(c_0)$ in the weak$^*$ topology. That is \begin{equation}\label{eq weak and strong star to weak star cont of the cont functional calculus} \hbox{If $(c_{\lambda})_{\lambda}$ is a bounded net of positive elements in $M$ and $(c_{\lambda})_{\lambda} \to c_0$}
\end{equation}
$$ \hbox{ in the strong$^*$ topology, then for each continuous function $f:[0,1]\to \mathbb{R}$}$$ $$
\hbox{ the net $(f(c_{\lambda}))_{\lambda}$ converges to $f(c_0)$ in the weak$^*$ topology.}$$

Suppose now that $(a_{\lambda})_{\lambda}$ is a net in $\mathcal{B}_{M}$ converging to $a_0$ in the strong$^*$ topology. It follows from the joint strong$^*$ continuity of the product that $(|a_{\lambda}|^2)_{\lambda} = (a_{\lambda}^*a_{\lambda})_{\lambda} \to a_0^* a_0 = |a_0|^2$ in the strong$^*$ topology of $M$. We deduce from \eqref{eq weak and strong star to weak star cont of the cont functional calculus} that $(|a_{\lambda}|)_{\lambda}\to |a_0|$ in the strong$^*$ topology. Fix now a normal state $\phi \in M_*$. The identity
\begin{equation*}
\begin{split}
\left\| |a_{\lambda}| - |a_0|\right\|_{\phi}^2  & = 2 \phi \left( (|a_{\lambda}| - |a_0|)^2 \right) \\
& = 2 \left(\phi \left( |a_{\lambda}|^2 \right) + \phi \left( |a_{0}|^2 \right) - \phi \left( |a_{\lambda}| |a_0| \right) - \phi \left( |a_0| |a_{\lambda}| \right) \right),
\end{split}
\end{equation*} together with the separate weak$^*$-continuity of the product of $M$ prove that $\left\| |a_{\lambda}| - |a_0|\right\|_{\phi}^2\to 0$. That is \begin{equation}\label{eq weak and strong star cont of the absolute value} \hbox{If $(a_{\lambda})_{\lambda}\to a_0$ in the strong$^*$ topology with $(a_{\lambda})_{\lambda}$ bounded,}
\end{equation} $$\hbox{ then $(|a_{\lambda}|)_{\lambda}\to |a_0|$ in the strong$^*$ topology.}$$

\begin{theorem}\label{t presevers on Cstar} Let $T: A\to B$ be a contractive linear operator between two C$^*$-algebras. Suppose $T$ preserves domain absolutely compatible elements. Then $T$ is a triple homomorphism. Furthermore, if $T$ is symmetric, then $\Phi = T^{**}(1) T$ is a Jordan\hyphenation{Jordan} $^*$-homomorphism and $T =  T^{**}(1) \Phi$. If we additionally assume that $T$ is positive, then $T$ is a Jordan $^*$-homomorphism.
\end{theorem}

\begin{proof} Let us fix an element $a$ in $\mathcal{A}$. Since, in general, we can not find an spectral resolution for $a$, we shall employ a device due to Akemann and G.K. Pedersen. Let $a = u |a|$ be the polar decomposition of $a$ in $A^{**}$, where $u$ is a partial isometry in $A^{**}$, which in general does not belong to $A$. It is further known that $u^*u$ is the range projection of $|a|$ ($r(|a|)$ in short), and for each $h\in C(\sigma(|a|)),$ the space of all continuous complex valued functions on $\sigma(|a|)\subseteq [0,1]$ with $h(0)=0,$ the element $u h(|a|)\in A$ (see \cite[Lemma 2.1]{AkPed77}).\smallskip
	
Let us fix a set of the form $\mathcal{C} =[\alpha,\beta]\cap \sigma(|a|)$ with $\alpha,\beta\in [0,1]$. For $n$ large enough we define two positive continuous functions $f_n,g_n: \sigma(|a|)\to \mathbb{C}$ given by
	$$f_n(t):=\left\{%
	\begin{array}{ll}
	1, & \hbox{if $t\in [\alpha, \beta]$} \\
	\hbox{affine}, & \hbox{in $[\alpha-\frac1n,\alpha]\cup [\beta, \beta+\frac1n]$} \\
	0, & \hbox{if $t\in [0,\alpha-\frac1n]\cup [\beta+\frac1n,1]$} \\
	\end{array}%
	\right.$$ and $$ g_n(t):=\left\{%
	\begin{array}{ll}
	1, & \hbox{if $t\in [\alpha+\frac1n, \beta-\frac1n]$} \\
	\hbox{affine}, & \hbox{in $[\alpha,\alpha+\frac1n]\cup [\beta-\frac1n,\beta]$} \\
	0, & \hbox{if $t\in [0,\alpha]\cup [\beta,1]$.} \\
	\end{array}%
	\right.$$ It is no hard to check that the elements $a_n := u f_n(|a|)$ and $b_n :=u g_n(|a|)$ belong to the closed unit ball of $A$ with $a_n \triangle_d b_n$ for $n$ large enough (compare Remark \ref{r new}). Therefore, $T(a_n) \triangle_d T(b_n)$ for $n$ large enough, because, by hypothesis, $T$ preserves domain absolutely compatible elements. Proposition \ref{p 00} implies that \begin{equation}\label{eq 0408 a}  2 |T(a_n)| \circ |T(b_n)| = |T(a_n)| + |T(b_n)| - \vert |T(a_n)| - |T(b_n)| \vert,
	\end{equation} for $n$ large enough.

	Clearly the sequences $(a_n)$ and $(b_n)$ converge in the strong$^*$ topology of $A^{**}$ to the partial isometries $u \chi_{[\alpha,\beta]\cap \sigma(|a|)}$ and $u \chi_{(\alpha,\beta)\cap \sigma(|a|)}$, respectively. As usually, we identify the bidual of the C$^*$-subalgebra of $A$ generated by $|a|$ with its weak$^*$ closure in $A^{**}$. Now, since $T^{**}: A^{**}\to B^{**}$ is strong$^*$-continuous, taking strong$^*$-limits in \eqref{eq 0408 a}, we deduce from \eqref{eq weak and strong star cont of the absolute value} that \begin{equation}\label{eq 0409 b} 2 |T^{**}(u \chi_{[\alpha,\beta]\cap \sigma(|a|)})| \circ |T^{**}(u \chi_{(\alpha,\beta)\cap \sigma(|a|)})| = |T^{**}(u \chi_{[\alpha,\beta]\cap \sigma(|a|)})|
	\end{equation}
	$$  + |T^{**}(u \chi_{(\alpha,\beta)\cap \sigma(|a|)})| - \left\vert |T^{**}(u \chi_{[\alpha,\beta]\cap \sigma(|a|)})| - |T^{**}(u \chi_{(\alpha,\beta)\cap \sigma(|a|)})| \right\vert.$$

	Since partial isometries of the form $u \chi_{(\alpha,\beta)\cap \sigma(|a|)}$ (respectively, $u \chi_{[\alpha,\beta]\cap \sigma(|a|)}$) can be approximated in the strong$^*$ topology by sequences of the form $(u \chi_{(\alpha+\frac1n,\beta-\frac1n)\cap \sigma(|a|)})_n$ and $(u \chi_{[\alpha+\frac1n,\beta-\frac1n]\cap \sigma(|a|)})_n$ (respectively, by sequences of the form $(u \chi_{(\alpha-\frac1n,\beta+\frac1n)\cap \sigma(|a|)})_n$ and $(u \chi_{[\alpha-\frac1n,\beta+\frac1n]\cap \sigma(|a|)})_n$), we can deduce from \eqref{eq 0409 b} that \begin{equation}\label{eq 0409 c} 2 |T^{**}(u \chi_{(\alpha,\beta)\cap \sigma(|a|)})|^2 = 2 |T^{**}(u \chi_{(\alpha,\beta)\cap \sigma(|a|)})|
	\end{equation} (respectively, \begin{equation}\label{eq 0409 d} 2 |T^{**}(u \chi_{[\alpha,\beta]\cap \sigma(|a|)})|^2 = 2 |T^{**}(u \chi_{[\alpha,\beta]\cap \sigma(|a|)})|.)
	\end{equation} The above identities assure that $T^{**}(u \chi_{(\alpha,\beta)\cap \sigma(|a|)})$ and $T^{**}(u \chi_{[\alpha,\beta]\cap \sigma(|a|)})$ are partial isometries. We can now apply similar arguments and the above conclusions to deduce that $T^{**}(u \chi_{(\alpha,\beta]\cap \sigma(|a|)})$ and $T^{**}(u \chi_{[\alpha,\beta)\cap \sigma(|a|)})$ are partial isometries for every $\alpha,\beta\in [0,1]$.\smallskip
	
	Fix now $\alpha_1<\beta_1<\alpha_2<\beta_2$ in $[0,1]$. For $n$ large enough we can consider the positive continuous functions $f_n,g_n: \sigma(|a|)\to \mathbb{C}$ defined by
	$$f_n^{\pm} (t):=\left\{%
	\begin{array}{ll}
	1, & \hbox{if $t\in [\alpha_1, \beta_1]$} \\
	\hbox{affine}, & \hbox{in $[\alpha_1-\frac1n,\alpha_1]\cup [\beta_1, \beta_1+\frac1n]$} \\
	\pm 1, & \hbox{if $t\in [\alpha_2, \beta_2]$} \\
	\hbox{affine}, & \hbox{in $[\alpha_2-\frac1n,\alpha_2]\cup [\beta_2, \beta_2+\frac1n]$} \\
	0, & \hbox{if $t\in [0,\alpha_1-\frac1n]\cup [\beta_1+\frac1n, \alpha_2-\frac1n]\cup [\beta_2+\frac1n,1]$} \\
	\end{array}%
	\right.$$ and $$ g_n^{\pm}(t):=\left\{%
	\begin{array}{ll}
	1, & \hbox{if $t\in [\alpha_1+\frac1n, \beta_1-\frac1n]$} \\
	\hbox{affine}, & \hbox{in $[\alpha_1,\alpha_1+\frac1n]\cup [\beta_1-\frac1n,\beta_1]$} \\
	\pm 1, & \hbox{if $t\in [\alpha_2+\frac1n, \beta_2-\frac1n]$} \\
	\hbox{affine}, & \hbox{in $[\alpha_2,\alpha_2+\frac1n]\cup [\beta_2-\frac1n,\beta_2]$} \\
	0, & \hbox{if $t\in [0,\alpha_1]\cup [\beta_1,\alpha_2]\cup [\beta_2,1]$.} \\
	\end{array}%
	\right.$$ By repeating the above reasonings to the partial isometries $u \chi_{[\alpha_1,\beta_1]\cap \sigma(|a|)}\pm u \chi_{[\alpha_2,\beta_2]\cap \sigma(|a|)}$ and $u \chi_{(\alpha_1,\beta_1)\cap \sigma(|a|)}\pm u \chi_{(\alpha_2,\beta_2)\cap \sigma(|a|)}$, and the functions $f_n^{\pm}$ and $g_n^{\pm}$, we deduce that the elements $$T^{**}(u \chi_{(\alpha_1,\beta_1)\cap \sigma(|a|)} \pm u \chi_{(\alpha_2,\beta_2)\cap \sigma(|a|)}), \ T^{**}(u \chi_{[\alpha_1,\beta_1]\cap \sigma(|a|)}\pm u \chi_{[\alpha_2,\beta_2]\cap \sigma(|a|)}),$$ $$\!T^{**}\!(u \chi_{(\alpha_1,\beta_1]\cap \sigma(|a|)}\pm u \chi_{(\alpha_2,\beta_2]\cap \sigma(|a|)}),\hbox {and } T^{**}(u \chi_{[\alpha_1,\beta_1)\cap \sigma(|a|)}\pm u \chi_{[\alpha_2,\beta_2)\cap \sigma(|a|)}),$$ all are partial isometries. This is enough to conclude that the elements $T^{**}(u \chi_{[\alpha,\beta)\cap \sigma(|a|)})$ and $T^{**}(u \chi_{[\beta,\gamma)\cap \sigma(|a|)})$ are partial isometries for every $\alpha<\beta<\gamma\in [0,1]$. Since the element $a$ can be approximated in norm by a finite positive linear combination of partial isometries of the form $u \chi_{[\alpha,\beta)\cap \sigma(|a|)}$, we prove, as in the proof of Proposition \ref{p presevers on von Neumann}, that $$T(\{a,a,a\}) = \{T(a),T(a),T(a)\},$$ witnessing that $T$ is a triple homomorphism.
\end{proof}

Next, we shall complete the conclusion of Theorem \ref{t presevers on Cstar} with a series of examples.

\begin{remark}\label{r examples star hom}{\rm Let us fix a two (unital) C$^*$-algebras $A$ and $B$. Suppose $\Phi: A\to B$ is a $^*$-homomorphism. Its is easy to check that $\Phi(1)$ is a projection satisfying that $1-\Phi(1)$ is orthogonal to every element in the image of $\Phi$. Let us take $a,b$ in $A$ with $a\triangle_d b$. Since $\Phi (|x|)= |\Phi(x)|$, for every  $x\in A$, we have $$\Big| 1- |\Phi(a)|-|\Phi(b)| \Big| +\Big|  |\Phi(a)|-|\Phi(b)| \Big| \ \ \ \ \ \ \ \ \ \ \ \ \ \ \ \ \ \ \ \ \ \ \ \ \ \ \ \ \ \ \ \ \ \ \ \ $$
		\begin{align*}
		&= \Big| 1 - \Phi(1) +\Phi(1) - \Phi(|a|)-\Phi(|b|) \Big| +\Big|  \Phi(|a|)-\Phi(|b|) \Big| \\
		&= 1 - \Phi(1) +\Phi( | 1-|a|-|b|| ) +\Phi (\Big|  |a|-|b| \Big|) \\
		&=1 - \Phi(1) +  \Phi( | 1-|a|-|b|| + \Big|  |a|-|b| \Big|) = 1-\Phi(1) +\Phi(1) = 1,
		\end{align*}
		witnessing that $\Phi$ preserves domain absolutely compatible elements.\smallskip
		
		Every $^*$-homomorphism between C$^*$-algebras preserves (range/domain) absolutely compatible elements.\smallskip
		
		Let $\Psi: A\to B$ be a $^*$-anti-homomorphism. If we pick $a,b$ in $A$ with $a\triangle_d b$ (respectively, $a\triangle_r b$), it is not hard to check that $$ \Big| 1- |\Psi(a)^*|-|\Psi(b)^*| \Big| +\Big|  |\Psi(a)^*|-|\Psi(b)^*| \Big| \ \ \ \ \ \ \ \ \ \ \ \ \ \ \ \ \ \ \ \ \ \ \ \ \ \ \ \ \ \ \ \ \ \ \ \ $$
		\begin{align*}
		&= \Big| 1 - \Psi(1) +\Psi(1) - \Psi(|a|)-\Psi(|b|) \Big| +\Big|  \Psi(|a|)-\Psi(|b|) \Big| \\
		&= 1 - \Psi(1) +\Psi( | 1-|a|-|b|| ) +\Psi (\Big|  |a|-|b| \Big|) \\
		&=1 - \Psi(1) +  \Psi( | 1-|a|-|b|| + \Big|  |a|-|b| \Big|) = 1-\Psi(1) +\Psi(1) = 1,
		\end{align*}
		(respectively, $$ \Big| 1- |\Psi(a)|-|\Psi(b)| \Big| +\Big|  |\Psi(a)|-|\Psi(b)| \Big| \ \ \ \ \ \ \ \ \ \ \ \ \ \ \ \ \ \ \ \ \ \ \ \ \ \ \ \ \ \ \ \ \ \ \ \ \ \ \ \ \ \ $$
		\begin{align*}
		&= \Big| 1 - \Psi(1) +\Psi(1) - \Psi(|a^*|)-\Psi(|b^*|) \Big| +\Big|  \Psi(|a^*|)-\Psi(|b^*|) \Big| \\
		&= 1 - \Psi(1) +\Psi( | 1-|a^*|-|b^*|| ) +\Psi (\Big|  |a^*|-|b^*| \Big|) \\
		&=1 - \Psi(1) +  \Psi( | 1-|a^*|-|b^*|| + \Big|  |a^*|-|b^*| \Big|) = 1-\Psi(1) +\Psi(1) = 1),
		\end{align*} witnessing that $\Psi(a)\triangle_r \Psi(b)$ (respectively, $\Psi(a)\triangle_d \Psi(b)$). Consequently, $\Psi$ preserves absolutely compatible elements. \smallskip

		Let us take now two unitary elements $u,v\in B(H)$ with $v u, v u\neq 1$. Then the mapping $\Psi: B(H)\to B(H),$ $\Psi(x):=u x v$ is a triple homomorphism, but it is not symmetric nor multiplicative. However, for each $x\in B(H)$, we have $|\Psi(x)|^2=\Psi (x)^* \Psi (x) = v^* x^* u^* u x v = v^* |x|^2 v= \Phi_1(|x|^2)$, where $\Phi_1 : B(H)\to B(H)$ is the $^*$-automorphism given by $\Phi_1 (x) = v^* x v$ ($x\in B(H)$). Given $a,b$ in $A$ with $a\triangle_d b$, it follows that $$ \Big| 1- |\Psi(a)|-|\Psi(b)| \Big| +\Big|  |\Psi(a)|-|\Psi(b)| \Big| \ \ \ \ \ \ \ \ \ \ \ \ \ \ \ \ $$ $$ = \Big| 1- \Phi_1(|a|)-\Phi_1(|b|) \Big| +\Big|  \Phi_1 (|a|)-\Phi_1(|b|) \Big| =1,$$ because $\Phi_1$ is a $^*$-automorphism. That is, $\Psi$ preserves domain absolutely compatible elements. We similarly prove that $\Psi$ preserves range absolutely compatible elements.\smallskip
		
		Contrary to the previous two examples, our third choice shows that the reciprocal implication in Theorem \ref{t presevers on Cstar} is not always valid. Let us consider $T: M_2(\mathbb{C})\to M_2(\mathbb{C})$, $T(a) = a^t$ the transpose of $a$. Clearly, $T$ is a $^*$-anti-automorphism and a triple isomorphism. Let us take two minimal partial isometries $e$ and $v$ such that $e^* e= |e| = \left( \begin{array}{cc}
		1 & 0 \\
		0 & 0 \\
		\end{array}
		\right)=p,$ $v^* v =|v|= \left( \begin{array}{cc}
		0 & 0 \\
		0 & 1 \\
		\end{array}
		\right) = 1-p$, $ee^* =|e^t|= 1-p$ and $v v^* =|v^t|= \frac{1}{2} \left(
		\begin{array}{cc}
		1 & 1 \\
		1 & 1 \\
		\end{array}
		\right) = r$. It is not hard to see that $$| 1-|e|-|v|| + \Big|  |e|-|v| \Big| = 1,\hbox{ and }| 1-|a^t|-|b^t|| + \Big|  |a^t|-|b^t| \Big| = \sqrt{2}\ 1, $$ therefore $T$ does not preserve domain absolutely compatible elements.
	}
\end{remark}

\begin{corollary}\label{c r presevers on Cstar} Let $T: A\to B$ be a contractive linear operator between two C$^*$-algebras. Suppose $T$ preserves range absolutely compatible elements. Then $T$ is a triple homomorphism. Moreover, if $T$ is symmetric, then $\Phi = T^{**}(1) T$ is a Jordan\hyphenation{Jordan} $^*$-homomorphism and $T =  T^{**}(1) \Phi$. If we additionally assume that $T$ is positive, then $T$ is a Jordan $^*$-homomorphism.
\end{corollary}

\begin{proof} It is easy to check that the linear mapping $S: A\to B$, $S(x) := T(x^*)^*$ preserves domain absolutely compatible elements. Theorem \ref{t presevers on Cstar} assures that $S$ (and hence $T$) is a triple homomorphism.
\end{proof}

The characterization of triple homomorphism in terms of preservers of absolutely compatible elements reads as follows.

\begin{corollary}\label{c AC presevers on Cstar} Let $T: A\to B$ be a contractive linear operator between two C$^*$-algebras. Then $T$ preserves absolutely compatible elements {\rm(}i.e., $a\triangle b\Rightarrow T(a)\triangle T(b)${\rm)} if, and only if, $T$ is a triple homomorphism.
\end{corollary}

\begin{proof} The ``only if'' implication follows from Theorem \ref{t presevers on Cstar}. To prove the ``if'' implication suppose that $T$ is a triple homomorphism. It follows from the separate weak$^*$-continuity of the product of $A^{**}$ and the weak$^*$-density of $A$ in $A^{**}$ (Goldstine's theorem) that $T^{**} : A^{**}\to B^{**}$ is a triple homomorphism. We can therefore assume that $A$ and $B$ are von Neumann algebras and $T$ is weak$^*$ continuous.\smallskip
	
	Let $e=T(1)$. Clearly, $e$ is a partial isometry in $B$, $T(A)\subseteq ee^* B e^* e = \{e,\{e,T(A),e\},e\} = \{T(1),\{T(1),T(A),T(1)\},T(1)\},$ $e^*e B e^* e$ is a von Neumann subalgebra of $B$, and the mapping $\Phi: A\to e^*e B e^* e$, $\Phi (a):= e^* T(a)$ is a Jordan $^*$-homomorphism. By \cite[Theorem 2.3]{Bresar89} or by \cite[Theorem 10]{Kad51b}, we can find two weak$^*$-closed ideals $I$ and $J$ of $A$ such that $A= I\oplus^{\infty} J$, $I\perp J$, $\Phi_1=\Phi|_{I}: I \to e^*e B e^* e$ is a $^*$-homomorphism and $\Phi_2=\Phi|_{J}: J \to e^*e B e^* e$ is a $^*$-anti-homomorphism, and $\Phi_1 (a)\perp \Phi_2(b)$, for every $a\in I$, $b\in J$. We have seen in Remark \ref{r examples star hom} that $\Phi_1$ and $\Phi_2$ both preserve absolutely compatible elements. Let us denote $p_1=\Phi_1(1)$, $p_2 =\Phi_2(1)$, and let $\pi_1$ and $\pi_2$ denote the natural projections of $A$ onto $I$ and $J$, respectively. Since $T(x) = e \Phi(x)$ for every $x\in A$, the identities
	$$|T(a)|^2 = T(a)^* T(a) = \Phi(a)^* e^* e \Phi (a) = \Phi (a^*) \Phi(a) $$ $$ = \Phi_1 (\pi_1( a^*)) \Phi_1(\pi_1(a)) + \Phi_2 (\pi_2( a^*)) \Phi_2(\pi_2(a)) = \Phi_1 (|\pi_1( a)|^2) + \Phi_2 (|\pi_2(a)^*|^2),$$ and
	$$|T(a)^*|^2 = T(a) T(a)^* = e \Phi(a) \Phi (a)^* e^* = e \Phi_1 (|\pi_1(a)^*|^2)  e^*+ e \Phi_2 (|\pi_2(a)|^2)  e^*,$$ hold for every $a\in A$. We observe that the mapping $x\mapsto e \Phi_1 (x)  e^* = \widetilde{\Phi}_1 (x)$ from $I$ into $e e^* B e e^*$ is a $^*$-homomorphism and $x\mapsto e \Phi_2 (x)  e^*= \widetilde{\Phi}_2 (x)$ from $J$ into $e e^* B e e^*$ is a $^*$-anti-homomorphism.\smallskip

If we take $a,b\in A$ with $a \triangle b,$ by applying that $\Phi_1$ and $\Phi_2$ have orthogonal images, we conclude, via Remark \ref{r examples star hom}, that  $$\Big|1-|T(a)|-|T(b)|\Big| +\Big| |T(a)|-|T(b)| \Big| \ \ \ \ \ \ \ \ \ \ \ \ \ \ \ \ \ \ \ \  \  \ \ \ \ \ \ \ \ \ \ \ \ \ \ \ \ \ \ \ \ \ \ \ \ \ \ $$ \vspace*{-4mm}
	\begin{align*} &= \Big|1-\Phi_1 (|\pi_1( a)|) - \Phi_2 (|\pi_2(a)^*|)- \Phi_1 (|\pi_1( b)|) - \Phi_2 (|\pi_2(b)^*|) \Big|  \\
	&+\Big|  \Phi_1 (|\pi_1( a)|) + \Phi_2 (|\pi_2(a)^*|) - \Phi_1 (|\pi_1( b)|) - \Phi_2 (|\pi_2(b)^*|) \Big| \\
	&= 1-(p_1+p_2) + \Big|p_1-\Phi_1 (|\pi_1( a)|) - \Phi_1 (|\pi_1( b)|) \Big| \\
	&+\Big|  \Phi_1 (|\pi_1( a)|) - \Phi_1 (|\pi_1( b)|) \Big|  + \Big|p_2- \Phi_2 (|\pi_2(a)^*|)- \Phi_2 (|\pi_2(b)^*|) \Big| \\
	&+\Big|  \Phi_2 (|\pi_2(a)^*|) - \Phi_2 (|\pi_2(b)^*|) \Big| \\
	&= 1-(p_1+p_2) + \Phi_1 \left(\Big|\pi_1(1) -|\pi_1( a)| - |\pi_1( b)| \Big|\right) \\
	& +\Phi_1 \left(\Big|  |\pi_1( a)| - |\pi_1( b)| \Big|\right)  + \Phi_2 \left(\Big|\pi_2(1) - |\pi_2(a)^*|- |\pi_2(b)^*| \Big|\right)   \\
	&+\Phi_2\left(\Big| |\pi_2(a)^*| - |\pi_2(b)^*| \Big|\right) = 1-p_1-p_2+p_1+p_2=1,
	\end{align*}
	and similarly,
	$$\Big|1-|T(a)^*|-|T(b)^*|\Big| +\Big| |T(a)^*|-|T(b)^*| \Big| \ \ \ \ \ \ \ \ \  \ \ \ \ \ \ \ \ \ \ \ \ \ \  \ \ \ \ \ \ \ \ \ \ \ \  \ \ \ \  \ \ \ \ $$ \vspace*{-4mm}
	\begin{align*}
	&= 1-(p_1+p_2) + \widetilde{\Phi}_1 \left(\Big|\pi_1(1) -|\pi_1( a)^*| - |\pi_1( b)^*| \Big|\right) \\
	&+\widetilde{\Phi}_1 \left(\Big|  |\pi_1( a)^*| - |\pi_1( b)^*| \Big|\right) + \widetilde{\Phi}_2 \left(\Big|\pi_2(1) - |\pi_2(a)|- |\pi_2(b)| \Big|\right) \\
	&+\widetilde{\Phi}_2\left(\Big| |\pi_2(a)| - |\pi_2(b)| \Big|\right) = 1-p_1-p_2+p_1+p_2=1.
	\end{align*} We have therefore shown that $T(a)\triangle T(b)$, which finishes the proof.
\end{proof}

Linear absolutely compatible preservers are naturally linked to another class of maps which has been well studied in the setting of preservers. Let $A$ and $B$ denote two C$^*$-algebras. Accordingly to the standard notation in \cite{Gardner1979,Gardner1979a} (see also \cite{GuanWangHou2015,Mol1996,Radja2003,RadjaSeddi2001} and \cite{Tagh2012}), we say that a linear map $\Phi: A\to B$ \emph{preserves absolute values} if $\Phi (|a|) = \left| \Phi(a)\right|$, for every $a\in A$. The detailed study published by L.T. Gardner in \cite{Gardner1979,Gardner1979a} shows that a bounded linear mapping $T: A\to B$ preserves absolute values if, and only if, it is 2-positive and preserves elements having zero product if, and only if, it is 2-positive and preserves positive elements having zero product if, and only if, there exists a positive element $b\in B^{**}$, and a $^*$-homomorphism $\Psi: A\to B^{**}$ such that $b$ is supported on $\displaystyle \overline{\bigcup_{a} \hbox{range}\Phi (a)}$ and centralizing $\Phi (A)$, with $\Phi (a) = b \Psi (a)$, for all $a\in A$. Triple homomorphisms between C$^*$-algebras preserve absolutely compatible elements but, in general, they do not preserve absolute values (compare Remark \ref{r examples star hom} and Corollary \ref{c AC presevers on Cstar}). Let $T: A\to B$ be a contractive linear map between unital C$^*$-algebras. It can be concluded that $T$ preserves absolutely compatible elements if $T$ preserves absolute values and $T(1)$ is a partial isometry.\smallskip

\textbf{Acknowledgements} Third author partially supported by the Spanish Ministry of Economy and Competitiveness (MINECO) and European Regional Development Fund project no. MTM2014-58984-P and Junta de Andaluc\'{\i}a grant FQM375.

\end{document}